\documentclass[letterpaper, 10 pt, conference]{ieeeconf}  

\IEEEoverridecommandlockouts                              

\usepackage{amsmath}
\usepackage{amssymb}
\usepackage[utf8]{inputenc}
\usepackage{eurosym}
\usepackage{graphicx}
\usepackage{hyperref}
\usepackage{dsfont}
\usepackage{dsfont}
\usepackage{xcolor,colortbl}
\usepackage{comment}
\usepackage{tikz}
\usepackage{subcaption}

\allowdisplaybreaks

\usepackage{hyperref}

\usepackage{ifthen}

\makeindex

\newcommand{\Rr}{{\mathbb{R}}}

\newcommand{\Ee}{{\mathds{E}}}

\newcommand{\Pp}{{\mathcal{P}}}

\newcommand{\bepsilon}{{\bm{ \epsilon}}}

\newcommand{\bX}{{\bf X}}

\newcommand{\bZ}{{\bf Z}}

\newcommand{\bP}{{\bf P}}

\newcommand{\bv}{{\bf v}}

\def\d{{\rm d}}
\def\dx{{\rm d}x}

\def\dt{{\rm d}t}

\def\leq{\leqslant}
\def\geq{\geqslant}

\newtheorem{theorem}{Theorem}
\newtheorem{problem}{Problem}

\newtheorem{pro}[theorem]{Proposition}
\newtheorem{remark}[theorem]{Remark}

\usepackage{bm}
\usepackage{algorithm2e}
\makeatletter
\renewcommand{\@algocf@capt@plain}{above}
\makeatother

\title{\LARGE \bf
Machine Learning architectures for price formation models with common noise
}

\author{
\thanks{King Abdullah University of Science and Technology (KAUST), CEMSE Division, Thuwal 23955-6900. Saudi Arabia. e-mail: diogo.gomes@kaust.edu.sa.}}
\author{
\thanks{King Abdullah University of Science and Technology (KAUST), CEMSE Division, Thuwal 23955-6900. Saudi Arabia. e-mail: julian.gutierrezpineda@kaust.edu.sa.}}
\author{
\thanks{}}

\author{Diogo Gomes$^{1}$, Julian Gutierrez$^{1}$, and Mathieu Lauri\`{e}re$^{2}$%
\thanks{$^{1}$King Abdullah University of Science and Technology (KAUST), CEMSE Division, Thuwal 23955-6900. Saudi Arabia
        {\tt\small diogo.gomes@kaust.edu.sa}
        {\tt\small julian.gutierrezpineda@kaust.edu.sa}}%
\thanks{$^{2}$New York University Shanghai (NYU Shanghai), Shanghai 200122. China
        {\tt\small ml5197@nyu.edu}}
}

\begin{document}

\maketitle
\thispagestyle{empty}
\pagestyle{empty}

\begin{abstract}
%


We propose a machine learning method to solve a mean-field game price formation model with common noise. This involves determining the price of a commodity traded among rational agents subject to a market clearing condition imposed by random supply, which presents additional challenges compared to the deterministic counterpart. Our approach uses a dual recurrent neural network architecture encoding noise dependence and a particle approximation of the mean-field model with a single loss function optimized by adversarial training. We provide a posteriori estimates for convergence and illustrate our method through numerical experiments.

\end{abstract}
\section{INTRODUCTION}

In this work, we extend the use of machine learning (ML) techniques for the numerical solution
of the mean-field games (MFGs) price formation models, introduced in \cite{MLpriceDet2022},
to incorporate the common noise model from \cite{GoGuRi2021} (see also \cite{gomes2021randomsupply}). 
The goal is to determine the price $\varpi$ of a commodity with a noisy supply $Q$ traded among rational agents within a finite time horizon $T > 0$, under a market-clearing condition.
More precisely, we assume the supply function $Q$ satisfies the following stochastic differential equation (SDE)
\begin{equation}
\label{eq:SDE supply}
	\begin{cases}
\d Q(t) = b^S(Q(t),t)\d t + \sigma^S(Q(t),t)\d W(t),
\\
Q(0)=q_0
	\end{cases}
\end{equation}
where $q_0 \in \Rr$ and $W$ is a one-dimensional Brownian motion acting as common noise. The coefficients $b^S$ and $\sigma^S$ satisfy the usual Lipschitz conditions for existence and uniqueness of solutions (see \cite{BackwardSDE1997}). Because of \eqref{eq:SDE supply}, our model explains the price formation for commodities with continuous and smooth fluctuations, such as stocks, bonds, currencies, and continuously produced or consumed goods such as oil or natural gas. Additional sources of noise can be considered. For instance, sudden and discontinuous fluctuations can be modeled by adding Poisson jumps to \eqref{eq:SDE supply}.

Let $(\Omega,\mathcal{F},\mathds{F},\mathds{P})$ be a complete filtered probability space supporting $W$. Progressive measurability refers to the measurability with respect to this filtration, which we require for all stochastic processes. In this context, the MFG with common noise characterizing the price is the following.
\begin{problem} 
\label{problem:MFGs price problem Cnoise} 
Suppose that $H:\Rr^2\to \Rr$ is uniformly convex and differentiable in the second argument, $m_0$ is a probability measure on $\Rr$, and $u_T:\Rr\to\Rr$ is uniformly convex and differentiable. Find $m: [0,T]\times \Rr \to \Rr$, $u,Z:  [0,T]\times \Rr \times \Omega \to \Rr$, and $\varpi:  [0,T]\times \Omega \to \Rr$ progressively measurable, satisfying  $m\geq 0$ and 
\begin{equation}\label{eq:MFG system Cnoise}
\begin{cases}
-\d u +H(x,\varpi + u_x ) \d t = Z(t,x) \d W(t), 
\\
u(T,x)=u_T(x), 
\\
m_t - \left(H_p(x,\varpi + u_x)m\right)_x =0, 
\\
m(0,x)=m_0(x), 
\\
-\int_{\Rr} H_p(x,\varpi + u_x)m \dx = Q(t).
\end{cases}
\end{equation}
\end{problem} 
\smallskip
The previous problem generalizes the one introduced in \cite{gomes2018mean}, which corresponds to the case $\sigma^S =0$. The numerical solution of \eqref{eq:MFG system Cnoise} presents additional challenges compared to the deterministic counterpart, as the state space becomes infinite-dimensional.  \cite{gomes2021randomsupply} showed that \eqref{eq:MFG system Cnoise} is well-posed when $b^S$ and $\sigma^S$ are linear and $H$ is quadratic, obtaining semi-explicit solutions. 
Section \ref{sec:The mfg price problem} presents the derivation of \eqref{eq:MFG system Cnoise}.

In the absence of common noise, several numerical schemes have been proposed: 
Fourier series \cite{Yang2021}, semi-Lagrangian schemes \cite{CarliniSilva2014}, fictitious play \cite{HADIKHANLOO2019369}, and variational methods \cite{BenamouBrenierVarMFG}.  \cite{dayanikli2023machine} proposes an ML-based approach to solve bi-level Stackelberg problems between a principal and a mean field of agents by reformulating the problem as a single-level mean-field optimal control problem. \cite{LauriereEtAlSurvey} and \cite{carmonalauriere2021handbookdeep} survey deep learning and reinforcement learning methods applied to MFGs and mean-field control problems. However, these methods cannot handle general forms of common noise as the state space becomes infinite-dimensional. Recent works have circumvented this issue. \cite{FiniteStateCNoise} reduces continuous-time mean field games with finitely many states and common noise to a system of forward-backward systems of (random) ordinary differential equations. \cite{min2021signatured} used rough path theory and deep learning techniques.  However, the coupling in the price formation problem with common noise is given by an integral constraint in infinite dimensional spaces, which is beyond what standard methods can handle. In \cite{IEEEPotential}, the price formation model with common noise was converted into a convex variational problem with constraints and solved using ML, enforcing constraints by penalization.  This approach, however, introduces numerical instabilities. In contrast, our method includes the balance constraint in the loss functional as a Lagrange multiplier instead of a penalization.

Our method employs two recurrent neural networks (RNN) to approximate $\varpi$ and the optimal vector field agents follow using a particle approximation and a loss function that the RNNs optimize by adversarial training. 
We develop a posteriori estimates 
to confirm the convergence of our method, which is of paramount importance when no benchmarks are available. Given $f:[0,T] \times \Omega$, we write $\| f\| = \left(\Ee \left[ \|f\|_{L^2([0,T]}^2 \right] \right)^{1/2}$. We introduce additional notation in Section \ref{sec:A posteriori estimates}, where we prove the following main result:
\begin{theorem}
\label{Thm:Main result}
Suppose that $H$ is uniformly concave-convex in $(x,p)$, separable, with Lipschitz continuous derivatives, and $u_T$ is convex with Lipschitz continuous derivative. Let $(\bX,\bP)$ and $\varpi^N$ solve the $N$-player price formation problem with common noise, and let $(\tilde{\bX},\tilde{\bP})$ and $\tilde{\varpi}^N$ be an approximate solution to the $N$-player problem up to the error terms $\epsilon_H$ and $\epsilon_B$. Then, there exists $C>0$, depending on problem data, such that 
\begin{align*}
	\| \varpi^N - \tilde{\varpi}^N \|  &  \leq C \bigg( \|\epsilon_H\| + \|\epsilon_B\|\bigg).
\end{align*}
\end{theorem}

\smallskip

We present our algorithm in Section \ref{sec:NN for price} and numerical results in Section \ref{sec: Numerical results} for the linear-quadratic setting. Nonetheless, our method can handle models outside the linear-quadratic framework. Moreover, the ML is well suited for higher-dimensional state spaces, where, for instance, several commodities are priced simultaneously. Section \ref{sec:Conclusions} contains concluding remarks and sketches future research directions.

\section{The mfg price problem with common noise}
\label{sec:The mfg price problem}

Price formation is a critical aspect of economic systems.
One example is load-adaptive pricing in smart grids, which motivates consumers to adjust their energy consumption based on changes in electricity prices. MFGs provide a mathematical framework for studying complex interactions between multiple agents, including buyers and producers in a market. Here, we revisit the underlying optimization problem in Problem \ref{problem:MFGs price problem Cnoise}. 

A representative player with an initial quantity $x_0\in \Rr$ at time $t=0$ selects a progressively measurable trading rate $v:[0,T]\times \Omega \to \Rr$ to minimize the cost functional mapping $v$ to
\begin{align}
\label{eq:Functional per agent 0}
& \Ee \left[ \int_0^T \left(L(X(t),v(t)) + \varpi(t) v(t) \right) \dt + u_T\left(X(T)\right) \right],
\end{align}
where $X$ solves
\begin{equation}
\label{eq:Agent dynamics}
\begin{cases}
\d X(t)=v(t) \dt,
\\
X(0)=x_0,
\end{cases}
\end{equation}
with $x_0 \sim m_0$. 
The Hamiltonian $H$ in \eqref{eq:MFG system Cnoise} is the Legendre transform of the Lagrangian $L$ in \eqref{eq:Functional per agent 0}:
\begin{equation}\label{eq:Legendre transform}
	H(x,p)=\sup_{v\in\Rr} \left\{-p v - L(x,v)\right\}, \quad (x,p) \in \Rr^2.
\end{equation}
Here, we assume that 
$v\mapsto L(x,v)$ is uniformly convex for all $x\in \Rr$. Moreover, we assume that $H$ satisfies the assumptions of Theorem \ref{Thm:Main result}, which guarantees the Lagrangian satisfies the convexity requirement. Considering the value function
\begin{align*}
	u(t,x) & = \Ee \left[ \int_t^T \left(L(X(s),v(s)) + \varpi(s) v(s) \right) \d s | W(t)  \right]
	\\
	& \quad + \Ee \left[ u_T\left(X(T)\right) \left. \right| W(t)  \right],
\end{align*}
we obtain the first equation in \eqref{eq:MFG system Cnoise}. The distribution starts with $m_0 \in \Pp(\Rr)$ and evolves according to the stochastic flow controlled by $v^*(t,x)=-H_p(x,\varpi(t)+u_x(t,x))$, as described by the second equation in \eqref{eq:MFG system Cnoise}, while the third equation imposes a market-clearing condition. \cite{IEEEPotential} discusses further details. In our method, we approximate $\varpi$, which allows the decoupling of the equations in \eqref{eq:MFG system Cnoise}.

The particle approximation involves a finite population of $N$ players with independent, identically distributed initial positions $x_0^n \in \Rr$, $n=1,\ldots,N$, with distribution $m_0$. Each player selects $v^n: [0,T] \times \mathbb{R} \times \Omega \to \mathbb{R}$, $1\leq n \leq N$, determining its trajectory $X^n$ according to \eqref{eq:Agent dynamics} and aiming at minimizing the functional mapping $v^n$ to
\begin{align}
\label{eq:Functional per agent}
& \Ee \left[ \int_0^T \big( L(X^n(t),v^n(t)) + \varpi^N(t) \left( v^n(t) - Q(t) \right) \big) \d t \right] \nonumber
\\
& + \Ee \left[ u_T\left(X^n(T)\right) \right].
\end{align}
The existence of ${v^*}^n$ minimizing \eqref{eq:Functional per agent} for $1\leq n \leq N$ corresponds to the existence of $\bv^*=({v^*}^1,\ldots,{v^*}^N)$ minimizing the functional
\begin{gather*}
\bv \mapsto \Ee \left[ \frac{1}{N} \sum_{n=1}^N  \int_0^T L(X^n(t),v^n(t))  \dt + u_T\left(X^n(T)\right) \right]
\end{gather*}
subject to the market-clearing constraint
\begin{equation}
\label{eq:Balance c}
	\frac{1}{N}\sum\limits_{n=1}^N v^n(t)-Q(t)=0.
\end{equation}
We rely on the existence and uniqueness result for the $N$-player price formation model, presented in \cite{gomes2021randomsupply}, which determines the price ${\varpi^*}^N:[0,T] \times \Omega \to\Rr$ in \eqref{eq:Functional per agent} through the Lagrange multiplier associated with the market-clearing constraint. Our goal is to extend the ML algorithm introduced in \cite{MLpriceDet2022} to cover the case of common noise, providing a solution to the price formation problem in random environments. Relying on the particle approximation of the model to approximate the price solving \eqref{eq:MFG system Cnoise}, we approximate stationary points of the functional mapping $(\bv,\varpi^N)$ to
\begin{align}
\label{eq:Saddle points functional}
& \Ee \left[ \frac{1}{N} \sum_{n=1}^N \int_0^T \big(L(X^n(t),v^n(t)) \right. \nonumber
\\
& \quad \quad \quad \left. + \varpi^N(t)\left(v^n(t) - Q(t) \right)\big) dt+ u_T\left(X^n(T)\right) \right]
\end{align}
by minimizing w.r.t. $\bv$ and maximizing w.r.t. $\varpi^N$. The approximation is done in the ML framework, and we guarantee its accuracy using a posteriori estimates of the $N$-player model, which we present next. 
\section{A posteriori estimates}
\label{sec:A posteriori estimates}
In this section, we use optimality conditions for the $N$-player game to obtain a posteriori
  estimates to verify our approximation's convergence. We extend the proof presented in \cite{MLpriceDet2022} to the common noise setting with minor modifications.

The optimality conditions for  \eqref{eq:Functional per agent}  give rise to a 
Hamiltonian system comprising the following backward-forward stochastic differential equation
\begin{equation}
\label{eq:HSystem single agent}
	\begin{cases}
	\d P^n(t) = H_x(X^n(t),P^n(t)+\varpi^N(t)) \d t 
	\\
	\quad \quad \quad \quad + Z^n(t) \d W(t),
	\\
	P^n(T) = u'_T(X^n(T)),
	\\
	\d X^n(t) = -H_p(X^n(t),P^n(t)+\varpi^N(t)) \d t,
	\\
	X^n(0) = x^n_0,
	\end{cases}
\end{equation}
for $1 \leq n \leq N$, where $H$ is given by \eqref{eq:Legendre transform}. Notice that $Z^n$ is part of the unknowns. Moreover, $\bv^*$ and ${\varpi^*}^N$ solving the $N$-player price formation problem define a solution of \eqref{eq:HSystem single agent} by
\begin{equation}
\label{eq:Adjoint relation}
	{P^*}^n(t) + L_v({X^*}^n(t),{v^*}^n(t)) + {\varpi^*}^N(t) =0
\end{equation}
for $1 \leq n \leq N$, defining a saddle point of \eqref{eq:Saddle points functional} that satisfies the market-clearing constraint \eqref{eq:Balance c}. Let $\tilde{P}^n$, $\tilde{Z}^n$, $\tilde{X}^n$, and $\tilde{\varpi}^N$ satisfy
\begin{equation}
\label{eq:HSystem single agent error}
	\begin{cases}
	\d \tilde{P}^n(t) = \left( H_x(\tilde{X}^n(t),\tilde{P}^n(t)+\varpi^N(t)) + \epsilon^n(t)\right) \d t 
	\\
	\quad \quad \quad \quad + \tilde{Z}^n(t) \d W(t),
	\\
	\tilde{P}^n(T) = u'_T(\tilde{X}^n(T)) - \epsilon_T^n,
	\\
	\d \tilde{X}^n(t) = -H_p(\tilde{X}^n(t),\tilde{P}^n(t)+\tilde{\varpi}^N(t)) \d t,
	\\
	\tilde{X}^n(0) = x^n_0,
	\\
	\frac{1}{N}\sum\limits_{n=1}^N -H_p(\tilde{X}^n(t),\tilde{P}^n(t)+\tilde{\varpi}^N(t)) = Q(t) + \epsilon_B(t),
	\end{cases}
\end{equation}
where $\epsilon^n, \epsilon_B: [0,T]\times \Omega \to \Rr$, $\epsilon_T^n:\Omega \to \Rr$, for $1 \leq n \leq N$. We write $\bX = (X^1,\ldots,X^N)$, and, analogously, for all $n$-indexed stochastic processes. We denote $\epsilon_H = (\epsilon^1,\ldots,\epsilon^N,\epsilon_T^1,\ldots,\epsilon_T^N)$ and $\mathds{1} = (1,\ldots,1)\in \Rr^N$.

\begin{pro}\label{lem:Ps bound}
Under the assumptions of Theorem \ref{Thm:Main result}, let $(\bX,\bP,\bZ)$ and $\varpi^N$ solve \eqref{eq:HSystem single agent}, and let $(\tilde{\bX},\tilde{\bP},\tilde{\bZ})$, $\tilde{\varpi}^N$, $\epsilon_H$, and $\epsilon_B$ satisfy \eqref{eq:HSystem single agent error}. Then,
\begin{align*}
	\|P^n - \tilde{P}^n\|^2 & \leq C \left( \|X^n(T) - \tilde{X}^n(T)\|^2 + \|X^n - \tilde{X}^n\|^2 \right. 
	\\
	& \left. \quad \quad \quad + \|\epsilon^n_T\|^2 + \|\epsilon^n\|^2 \right),
\end{align*}
for $1\leq n \leq N$, where $C>0$ depends on problem data.
\end{pro}
\begin{proof}
Integrating on $[t,T]$ the first equations in \eqref{eq:HSystem single agent} and \eqref{eq:HSystem single agent error}, taking expectations, and using the martingale property of the processes $Z^n$ and $\tilde{Z}^n$, we have the estimate
\begin{align*}
	\Ee \left[ (P^n - \tilde{P}^n)^2 \right] & \leq C \Ee \left[ (X^n(T) - \tilde{X}^n(T))^2 \right.
	\\
	& \left. \quad + (X^n - \tilde{X}^n)^2 + (\epsilon^n_T)^2 + (\epsilon^n)^2 \right]
\end{align*}
for $1\leq n \leq N$. Integrating the previous inequality over $[0,T]$ we obtain the result. 
\end{proof}

\begin{pro}\label{Prop:PricevsELError} 
Under the assumptions of Theorem \ref{Thm:Main result}, let $(\bX,\bP,\bZ)$ and $\varpi^N$ solve \eqref{eq:HSystem single agent}, and let $(\tilde{\bX},\tilde{\bP},\tilde{\bZ})$, $\tilde{\varpi}^N$, $\epsilon_H$, and $\epsilon_B$ satisfy \eqref{eq:HSystem single agent error}. Then,
\begin{align*}
	& \|\bP + \mathds{1}\varpi^N - (\tilde{\bP} + \mathds{1} \tilde{\varpi}^N)\|^2 + \| \bX - \tilde{\bX} \|^2  
	\\
	& \leq C \left( \|\epsilon_H\|^2 + +\|\epsilon_B \|^2 \right),
\end{align*}
where $C>0$ depends on problem data.
\end{pro}

\begin{proof}
We write $\|\cdot \|_2 = \|\cdot \|_{L^2([0,T])}$. The uniform concavity-convexity assumption on $H$ and the equations in \eqref{eq:HSystem single agent} and \eqref{eq:HSystem single agent error} give
\begin{align*}
	& \tfrac{\gamma_p}{2}\| P^n + \varpi^N - (\tilde{P}^n +\tilde{\varpi}^N) \|_2^2 + \tfrac{\gamma_x}{2}\|X^n - \tilde{X}^n\|_2^2 
	\\
	& \leq \int_0^T \bigg( \d\left( \tilde{X}^n - X^n \right) ( P^n - \tilde{P}^n) 
	\\
	& \quad\quad\quad\quad - \d \left( X^n - \tilde{X}^n \right) ( \varpi^N -\tilde{\varpi}^N)
	\\
	& \quad\quad\quad\quad -\left( \d (P^n - \tilde{P}^n) + \epsilon^n \d t \right) ( X^n - \tilde{X}^n) 
	\\
	& \quad\quad\quad\quad + (Z^n - \tilde{Z}^n) \d W(t) \bigg)
\end{align*}	
for $1\leq n \leq N$, for some $\gamma_p,\gamma_x>0$. Using It\^{o} product rule, the initial and terminal conditions in \eqref{eq:HSystem single agent} and \eqref{eq:HSystem single agent error}, and the convexity of $u_T$, the previous inequality gives
\begin{align*}
	& \tfrac{\gamma_p}{2}\| P^n + \varpi^N - (\tilde{P}^n +\tilde{\varpi}^N) \|_2^2 + \tfrac{\gamma_x}{2}\|X^n - \tilde{X}^n\|_2^2 
	\\
	& \leq \left( -\tfrac{\gamma_T}{2} + \tfrac{1}{4 \delta_1} \right) ( X^n(T) - \tilde{X}^n(T))^2 + \delta_1 (\epsilon_T^n)^2 
	\\
	& \quad - \int_0^T \bigg(  \d \left( X^n - \tilde{X}^n \right) ( \varpi^N -\tilde{\varpi}^N) + \epsilon^n( X^n - \tilde{X}^n) \d t 
	\\
	& \quad + (Z^n - \tilde{Z}^n) \d W(t) \bigg)
\end{align*}
for some $\gamma_T>0$ and $\delta_1>0$ to be selected. Adding the previous inequality over $n$, and using the third equation in \eqref{eq:HSystem single agent} and \eqref{eq:HSystem single agent error}, we get
\begin{align}
\label{eq:Aux 1}
	& \tfrac{\gamma_p}{2}\| \bP + \mathds{1}\varpi^N - (\tilde{\bP} +\mathds{1} \tilde{\varpi}^N) \|_2^2 + \tfrac{\gamma_x}{2}\|\bX - \tilde{\bX}\|_2^2  \nonumber
	\\
	& \leq \left( -\tfrac{\gamma_T}{2} + \tfrac{1}{4 \delta_1} \right) |\bX(T) - \tilde{\bX}(T)|^2 + \delta_1 |\epsilon_T|^2 \nonumber
	\\
	& \quad + \int_0^T \bigg( N \epsilon_B ( \varpi^N -\tilde{\varpi}^N) \d t	+ (\bZ - \tilde{\bZ}) \d W(t) \bigg) \nonumber
	\\
	& \quad + \delta_2 \|\bepsilon\|_2^2 + \tfrac{1}{4\delta_2}\|\bX-\tilde{\bX}\|_2^2 
\end{align}
for some $\delta_2>0$ to be selected. By the triangle inequality,
\begin{equation}
\label{eq:Triangle 1}
	\frac{N}{2}\| \varpi^N - \tilde{\varpi}^N \|^2  \leq \|\bP + \mathds{1} \varpi^N - (\tilde{\bP} + \mathds{1}\tilde{\varpi}^N) \|^2 + \|\bP - \tilde{\bP}\|^2.
\end{equation}
Using \eqref{eq:Triangle 1} on the RHS of \eqref{eq:Aux 1}, taking expectations, and using Proposition \ref{lem:Ps bound}, we obtain
\begin{align*}
	& \tfrac{\gamma_p}{2}\| \bP + \mathds{1}\varpi^N - (\tilde{\bP} +\mathds{1} \tilde{\varpi}^N) \|^2 + \tfrac{\gamma_x}{2}\|\bX - \tilde{\bX}\|^2  \nonumber
	\\
	& \leq \left( -\tfrac{\gamma_T}{2} + \tfrac{1}{4 \delta_1} + \tfrac{C}{\delta_4}\right) \|\bX(T) - \tilde{\bX}(T)\|^2 
	\\
	& \quad + \left( \delta_1 +\tfrac{C}{\delta_4}\right)\|\epsilon_T\|^2 + \left( N \delta_3 +N \delta_4\right) \|\epsilon_B\|^2 \nonumber
	\\
	& \quad + \tfrac{1}{4 \delta_3} \|\bP + \mathds{1}\varpi^N -(\tilde{P} + \mathds{1}\tilde{\varpi}^N)\|^2 
	\\
	& \quad +\left( \delta_2 + \tfrac{C}{\delta_4} \right) \|\bepsilon\|^2 + \left( \tfrac{1}{4\delta_2} + \tfrac{C}{\delta_4}\right)\|\bX-\tilde{\bX}\|^2
\end{align*}
for some $\delta_3,\delta_4>0$ to be selected. Selecting $\delta_i$, $i=1,\ldots,4$ conveniently, the previous expression provides the result. 
\end{proof}
\begin{proof}[Proof of Theorem \ref{Thm:Main result}]
Using \eqref{eq:Triangle 1}, Proposition \ref{lem:Ps bound} and Proposition \ref{Prop:PricevsELError}, 
we get
\begin{align}
\label{eq:Triangle 2}
	& \frac{N}{2}\| \varpi^N - \tilde{\varpi}^N \|^2  \nonumber
	\\
	& \leq C\left( \|\epsilon_H\|^2 + \|\epsilon_B\|^2 + \|\bX(T) - \tilde{\bX}(T)\|^2 \right).
\end{align}
The Lipstichz continuity of $H_p$ and Proposition \ref{lem:Ps bound} give
\begin{align}
\label{eq:Aux T}
\|\bX(T) - \tilde{\bX}(T)\|^2 \leq C \left( \|\epsilon_H\|^2 + \|\epsilon_B\|^2 \right).
\end{align}
Using \eqref{eq:Aux T} in \eqref{eq:Triangle 2}, we obtain the result.
\end{proof}

\section{Neural Networks for progressively measurable processes}
\label{sec:NN for price}

This section details the RNN architectures we use to estimate $v^*$ and $\varpi$. Section \ref{sec: Numerical results} presents some  numerical experiments. 
RNNs, commonly used in natural language processing, generate outputs that sequentially depend on inputs. This architecture has a cell that iterates through input sequences and has a hidden state tracking historical dependencies; see
\cite{HighamML} for details. RNNs have also been used in the context of control problems with delay in~\cite{han2021recurrent}, but here our motivation comes from the impact of the common noise on the mean-field term. 

In our architecture, the RNN takes as inputs an ordered sequence, such as a discrete realization of the supply $\mathrm{Q}=\left(\mathrm{Q}^{\langle 0 \rangle},\ldots,\mathrm{Q}^{\langle K \rangle }\right)$. 
The RNN features a hidden state $\mathrm{h}$, initialized as zero, that captures the temporal dependence. Inside the RNN, a weight matrix $\mathrm{W}^{[l]}$ and a bias vector $\mathrm{b}^{[l]}$ determine layer $l$, where $1\leq l \leq L$. 
Their dimensions depend on the number of neurons $n^{[l]}$ per layer. 
The activation function of layer $j$ is denoted by $\sigma^{[l]}$. The cell parameters (weight matrices and bias vectors) are denoted by $\Theta$.

We use two RNNs for approximating the control variable $v$ and the price $\varpi$.
As usual in the ML framework, a trade-off must be made between computational cost and accuracy. Deep-RNN employs several layers and neurons in their cell. After multiple numerical experiments, we select $L=5$ layers, with $n^{[1]}=16$, $n^{[2]},n^{[3]},n^{[4]}=32$, and $n^{[5]}=1$ for the RNN approximating $\mathrm{v}^{\langle k\rangle }$, and $n^{[1]}, n^{[2]},n^{[3]},n^{[4]}=16$, and $n^{[5]}=1$ for the RNN approximating $\varpi$. As a common practice for RNN, the activation functions are hyperbolic tangent for the first layer, which computes the hidden state, and sigmoid for layers two to four. The last layer has an activation function equal to the identity, representing any real number as an output. Although we do not address it, an interesting research question is how sensitive the results are to the choice of parameters of the RNN. Moreover, a comparison regarding the accuracy and computational efficiency against other methods, such as forward-backward SDEs methods, can be formulated based on the adaptability of those methods to the price formation MFG problem with common noise.

We denote by $\Delta t =T/K$ the time-step size and discretize \eqref{eq:Agent dynamics} according to 
\begin{equation}\label{eq:NN Forward dynamics}
	X^{\langle k+1 \rangle}=X^{\langle k \rangle}+  \mathrm{v}^{\langle k \rangle}(\Theta_v) \Delta t, \quad X^{\langle 0 \rangle}=x_0
\end{equation}
for $k=0,\ldots,K$, where $\Theta_v$ is the parameter of the RNN approximating $v$. The second RNN, with parameter $\Theta_\varpi$, computes $\varpi^{\langle k \rangle}$ for $k=0,\ldots,K$. More precisely, the inputs and outputs of the two RNNs are as follows. For the RNN computing $\varpi(\Theta_\varpi)$, the input consists of a supply realization and the time; that is, $\left((Q^{\langle k \rangle})_{k =0}^{K}, \quad (t^{\langle k \rangle})_{k =0}^{K}\right)$. The output is $(\varpi^{\langle k \rangle})_{k =0}^{K}$. For the RNN computing $v(\Theta_v)$, the input consists of the time, the state variables (which the RNN updates according to \eqref{eq:NN Forward dynamics} as it iterates in the temporal direction), and the current price approximation; that is, $\left( (t^{\langle k \rangle})_{k =0}^{K}, (X^{\langle k \rangle})_{k =0}^{K}, (\varpi^{\langle k \rangle})_{k =0}^{K}, \right)$. The output is $(\mathrm{v}^{\langle k \rangle})_{k =0}^{K}$. Because we consider a population of $N$ agents, we add the superscript $(n)$ to denote the position and control sequence of the agent being considered; that is, $(X^{(n)\langle k \rangle})_{k =0}^{K}$, and $(\mathrm{v}^{(n)\langle k \rangle})_{k =0}^{K}$, for $1 \leq n \leq N$. 

\subsection{Numerical implementation of a posteriori estimates}
Let $\Delta \tilde{P}^{\langle k \rangle} = \tilde{P}^{\langle k+1 \rangle} - \tilde{P}^{\langle k \rangle}$ for $k=0,\ldots,K-1$. At the discrete level, \eqref{eq:HSystem single agent error} is equivalent to
\begin{equation*}
	\begin{cases}
	\Delta \tilde{P}^{(n)\langle k \rangle} 
	\\
	= \left( H_x(\tilde{X}^{(n)\langle k \rangle},\tilde{P}^{(n)\langle k \rangle}+\tilde{\varpi}^{N \langle k \rangle}) + \epsilon^{(n)\langle k \rangle}\right) \Delta t 
	\\
	\quad + Z^{(n)\langle k \rangle} \Delta W^{\langle k \rangle},
	\\
	\tilde{P}^{(n)\langle K \rangle} = u'_T(\tilde{X}^{(n)\langle K \rangle}) - \epsilon_T^{(n)},
	\\
	\Delta \tilde{X}^{(n)\langle k \rangle} = -H_p(\tilde{X}^{(n)\langle k \rangle},\tilde{P}^{(n)\langle k \rangle}+\tilde{\varpi}^{N \langle k \rangle}) \Delta t,
	\\
	\tilde{X}^{(n)\langle 0 \rangle} = x^n_0,	
	\\
	\frac{1}{N}\sum\limits_{n=1}^N -H_p(\tilde{X}^{(n)\langle k \rangle},\tilde{P}^{(n)\langle k \rangle}+\tilde{\varpi}^{N \langle k \rangle}) = Q^{\langle k \rangle} + \epsilon_B^{\langle k \rangle}
	\end{cases}
\end{equation*}
for $1 \leq n \leq N$. By \eqref{eq:Legendre transform}, $H_x(x,p) = -L_x(x,v)$ at the point $v$ where the supremum is achieved. Therefore, taking expectations on both sides of the equation in the previous system, and using the martingale property for the processes $\tilde{Z}^n$, for $1 \leq n \leq N$, we get
\begin{equation}
\label{eq:HSystem E Delta2}
	\begin{cases}
	\Ee\left[\Delta \tilde{P}^{(n)\langle k \rangle}\right] 
	\\
	= \Ee\left[ -L_x(\tilde{X}^{(n)\langle k \rangle},\tilde{v}^{(n)\langle k \rangle}) +  \epsilon^{(n)\langle k \rangle}\right] \Delta t 
	\\
	\Ee\left[ \tilde{P}^{(n)\langle K \rangle} \right]= \Ee\left[ u'_T(\tilde{X}^{(n)\langle K \rangle}) - \epsilon_T^{(n)} \right],		
	\\
	\Ee\left[\Delta \tilde{X}^{(n)\langle k \rangle}\right] = \Ee\left[\tilde{v}^{(n)\langle k \rangle}\right] \Delta t,
	\\
	\tilde{X}^{(n)\langle 0 \rangle} = x^n_0,			
	\\
	\Ee \left[ \frac{1}{N}\sum\limits_{n=1}^N \tilde{v}^{(n)\langle k \rangle} \right] = \Ee \left[ Q^{\langle k \rangle} + \epsilon_B^{\langle k \rangle}\right],
	\end{cases}
\end{equation}
where $\tilde{v}^{(n)\langle k \rangle} = -H_p(\tilde{X}^{(n)\langle k \rangle},\tilde{P}^{(n)\langle k \rangle}+\tilde{\varpi}^{N \langle k \rangle})$ drives the process $\tilde{X}^{(n)\langle k \rangle}$ according to \eqref{eq:NN Forward dynamics}. While the initial condition $\tilde{X}^{(n)\langle 0 \rangle}$ is deterministic, the terminal condition $\tilde{P}^{(n)\langle K \rangle}$ is random. We take a Monte Carlo (MC) approximation of \eqref{eq:HSystem E Delta2} with $J$ realizations; that is,
\begin{equation*}
	\begin{cases}
	\frac{1}{J}\sum\limits_{j=1}^J \Delta \tilde{P}^{(n)\langle k \rangle}_j 
	\\
	= \frac{\Delta t}{J}\sum\limits_{j=1}^J \left( -L_x(\tilde{X}^{(n)\langle k \rangle}_j,\tilde{v}^{(n)\langle k \rangle}_j) + \epsilon^{(n)\langle k \rangle}_j\right) 
	\\
	\frac{1}{J}\sum\limits_{k=1}^J \tilde{P}_j^{(n)\langle K \rangle} = \frac{1}{J}\sum\limits_{k=1}^J \left( u'_T(\tilde{X}_j^{(n)\langle K \rangle}) - {\epsilon_T}_j^{(n)} \right),			
	\\
	\frac{1}{J}\sum\limits_{k=1}^J \Delta \tilde{X}^{(n)\langle k \rangle}_j =  \frac{\Delta t}{J}\sum\limits_{j=1}^J \tilde{v}^{(n)\langle k \rangle}_j,
	\\
	\tilde{X}_j^{(n)\langle 0 \rangle} = x^n_0,			
	\\
	\frac{1}{J N}\sum\limits_{k=1}^J \sum\limits_{n=1}^N\tilde{v}^{(n)\langle k \rangle}_j = \frac{1}{J}\sum\limits_{j=1}^J \left( Q^{\langle k \rangle}_j + {\epsilon_B}_j^{\langle k \rangle} \right).
	\end{cases}
\end{equation*}
Thus, to implement the a posteriori estimate of Theorem \ref{Thm:Main result} numerically, let  $\tilde{v}^{(n)\langle k \rangle}_j$ and $\tilde{\varpi}^{N \langle k \rangle}$ be given. Define $\tilde{X}^{(n)\langle k \rangle}_j$ and $\tilde{P}^{(n)\langle k \rangle}_j$ according to \eqref{eq:NN Forward dynamics} and \eqref{eq:Adjoint relation}, respectively, and compute the mean-square error (MSE) of $\epsilon_H$ and $\epsilon_B$ by
\begin{align}
\label{eq:MSE HSystem and Balance}
	& \mbox{MSE}(\epsilon_H) \nonumber
	\\
	& =  \tfrac{1}{J N K}\sum_{j=1}^J\sum_{k=0}^K\sum_{n=1}^N \bigg(\bigg( \Delta \tilde{P}^{(n)\langle k \rangle}_j(t) \nonumber
	\\
	& \quad\quad\quad\quad\quad\quad\quad\quad\quad\quad  + \Delta t L_x(\tilde{X}^{(n)\langle k \rangle}_j(t),v^{(n)\langle k \rangle}_j(t))  \bigg)^2 \nonumber
	\\
	& \quad\quad\quad\quad\quad\quad\quad\quad\quad\quad  +\left( u'_T(\tilde{X}_j^{(n)\langle K \rangle}) - \tilde{P}_j^{(n)\langle K \rangle}	 \right)^2 \bigg), \nonumber
	\\ 
	& \mbox{MSE}(\epsilon_B) =  \tfrac{1}{J K}\sum_{j=1}^J\sum_{k=0}^K \left( \frac{1}{N}\sum_{n=1}^N v^{(n)\langle k \rangle}_j(t)  - Q^{\langle k \rangle}_j(t) \right)^2.
\end{align}

We measure \eqref{eq:MSE HSystem and Balance} as we train the neural network with the algorithm we introduce next. 

\begin{remark}
Theorem \ref{Thm:Main result} addresses the convergence for finite populations. A complete analysis of the convergence in our method involves three steps we identify by writing
\begin{align*}
\| \varpi - \tilde{\varpi}^N_{\mbox{\tiny ML}} \|  \leq \| \varpi - \varpi^N\| + \|\varpi^N - \varpi^N_{\mbox{\tiny MC}}\| + \| \varpi^N_{\mbox{\tiny MC}} - \tilde{\varpi}^N_{\mbox{\tiny ML}} \|.
\end{align*}
First is the convergence of finite to continuum population games. Second, the convergence of the MC approximation to the finite population game, addressing the dependence of sample size w.r.t. the population size. Third is the convergence of the ML to the MC approximation, involving the RNN parameters in the estimates. This is the error that our a posteriori estimate controls. 
\end{remark}

\subsection{Training algorithm}

In typical ML frameworks, a class of neural networks is trained by minimizing a loss function $\mathcal{L}$. Within a fixed architecture, $\mathcal{L}$ assigns a real number $\mathcal{L}(\Theta)$ to a parameter $\Theta$. The objective is to minimize $\mathcal{L}$ across the parameters $\Theta$. For a given realization $Q_i$ of the supply, the loss function is
\begin{align}\label{eq:Loss adversarial}
& \mathcal{L}\left( \Theta_{v},\Theta_{\varpi}\right) \nonumber
\\
& = \frac{1}{N} \sum_{n=1}^N \Bigg( \sum_{k=0}^{K-1}   \Delta_t \Big( L(X^{(n)\langle k \rangle},\mathrm{v}^{(n)\langle k \rangle}(\Theta_{v}))  \nonumber
\\
& \quad\quad\quad + \varpi^{\langle k \rangle}(\Theta_{\varpi})\left( \mathrm{v}^{(n)\langle k \rangle}(\Theta_{v}) - Q^{\langle k \rangle}_i\right)\Big) \nonumber
\\
& \quad\quad\quad + u_T(X^{(n)\langle M \rangle})\Bigg) .
\end{align}
 
The training algorithm is the following. 

\begin{algorithm}
\begin{tiny}
	\caption{Training algorithm}
    \SetKwInOut{Input}{Input}
    \SetKwInOut{Output}{Output}
    \Input{number of training iterations $I$, epoch size $I_e$, number of time steps $K$, training sample size $N_{\mbox{\tiny train}}$, test sample size $N_{\mbox{\tiny test}}$, MC sample size $J$, initial density $m_0$.}
    Initialize $\Theta_v^1,\Theta_\varpi^1$\; 
	\For{$i=1,\ldots,I$}{
	sample $(x_0^n)_{n=1}^{N_{\mbox{\tiny train}}}$ according to $m_0$\;
	sample $(Q^{\langle k \rangle}_i)_{k =0}^{K}$ according to \eqref{eq:SDE supply}\;
		compute $(\varpi^{\langle k \rangle}(\Theta_\varpi^i))_{k=0}^K$ and $((\mathrm{v}^{(n)\langle k \rangle}(\Theta_v^i))_{n=1}^{N_{\mbox{\tiny train}}})_{k=0}^K$\;
		compute $\mathcal{L}(\Theta_v^i,\Theta_\varpi^i)$ according to \eqref{eq:Loss adversarial}\;
	compute $\Theta_v^{i+1}$ by updating $\Theta_v^{i}$ in the descent direction $\mathcal{L}_{\Theta_v}(\Theta_v^i,\Theta_\varpi^i)$\;
	compute $((\mathrm{v}^{(n)\langle k \rangle}(\Theta_v^{i+1}))_{n=1}^{N_{\mbox{\tiny train}}})_{k=0}^K$\;
	compute $\mathcal{L}(\Theta_v^{i+1},\Theta_\varpi^i)$ according to \eqref{eq:Loss adversarial}\;		
	compute $\Theta_\varpi^{i+1}$ by updating $\Theta_\varpi^{i}$ in the ascent direction $\mathcal{L}_{\Theta_\varpi}(\Theta_v^{i+1},\Theta_\varpi^i)$\;
	\If{$i \mod I_e =0 $ (epoch is completed)}{
	sample $(x_0^n)_{n=1}^{N_{\mbox{\tiny test}}}$ according to $m_0$\;
	sample $((Q^{\langle k \rangle}_j)_{k =0}^{K})_{j=1}^{J}$ according to \eqref{eq:SDE supply}\;
	compute $(\varpi^{\langle k \rangle}(\Theta_\varpi^{i+1}))_{k=0}^K$ and $((\mathrm{v}^{(n)\langle k \rangle}(\Theta_v^{i+1}))_{n=1}^{N_{\mbox{\tiny test}}})_{k=0}^K$\;
	compute $\mbox{MSE}(\epsilon_H)$ and $\mbox{MSE}(\epsilon_B)$ according to \eqref{eq:MSE HSystem and Balance}.
	
	}
	}
    \Output{$\Theta_\varpi^I$, $\Theta_v^I$}
\label{alg:our algorithm}
\end{tiny}
\end{algorithm}

In contrast to the training algorithm in \cite{MLpriceDet2022}, in Algorithm~\ref{alg:our algorithm}, the supply input changes between training steps. The algorithm trains two neural networks in an adversarial manner. At each step, we generate a sample of the probability distribution $m_0$. To minimize the agent's cost function, we update $\Theta_v$ in the direction of descent while $\Theta_\varpi$ is fixed. Conversely, to penalize deviation from the balance condition, we maximize the cost functional by updating $\Theta_\varpi$ in the direction of ascent while $\Theta_v$ is fixed. This process is repeated multiple times, approximating the saddle point corresponding to the control minimizing the cost functional and its Lagrange multiplier.

\section{Numerical results}
\label{sec: Numerical results}
Here, we demonstrate how the a posteriori estimate (Theorem \ref{Thm:Main result}) ensures that our method delivers accurate price approximations. We validate our findings using the benchmarks provided by the linear-quadratic model. For numerical implementation, we employ the Tensorflow software library.

We set $T=1$ and $K=40$ for the time discretization. We assume that the supply follows
\begin{equation}
\label{eq:SDE supply Numerical}
	\d Q(t) = \left( b^S(t)-Q(t)\right)\d t + \sigma^S(t) \d W(t) , \quad Q(0)=0,
\end{equation}
where $b^S(t) = 3 \sin(3\pi t)$ and $\sigma^S(t) = \max\{ 0.5 \sin(2 \pi (t - 0.25)), 0\}$. The Brownian noise is applied on the time interval $[0.25,0.75]$ and generates deviations from the expected value, as illustrated in Figure \ref{fig:Q 2 samples} with two sample paths of the supply. The initial distribution $m_0$ is a normal distribution with mean $\overline{m}_0=-\tfrac{1}{4}$ and standard deviation $0.2$. The sample size for the training is $N_{\mbox{\tiny train}} = 30$. We train for $20$ epochs, an epoch consisting of $500$ training steps. We compute the MC estimate of the a posteriori estimate at the end of each epoch using $J = 60$ supply samples and a population size of $N_{\mbox{\tiny test}} = 30$. Empirically, the previous training parameters solved the trade-off between computational cost and accuracy.

\begin{figure}[htp]
     \centering
     \begin{subfigure}[t]{0.235\textwidth}
         \centering
         \includegraphics[width=\textwidth]{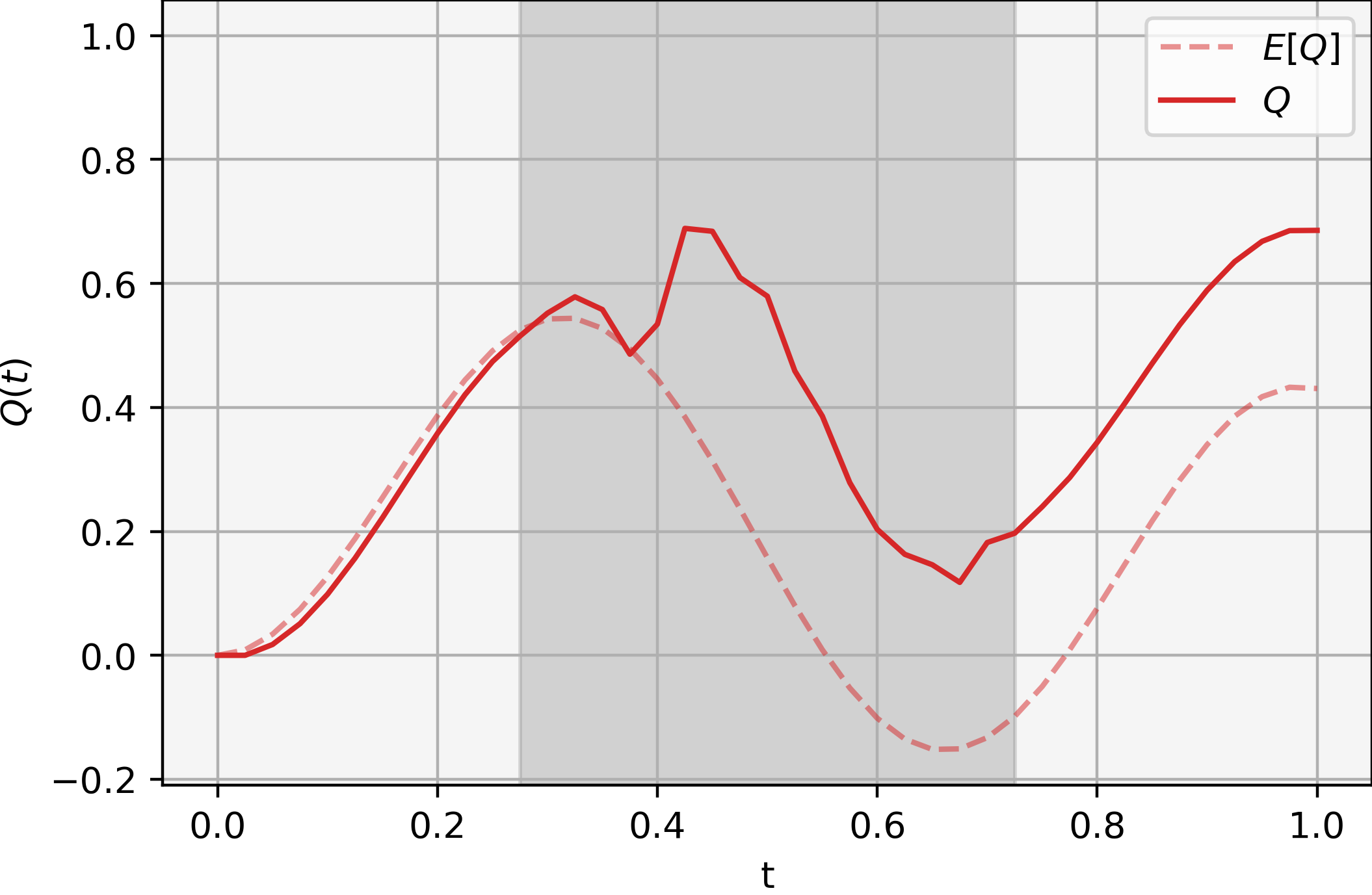}
         \caption{Supply realization.}
         \label{fig:Q1}
     \end{subfigure}
     \hfill
     \begin{subfigure}[t]{0.235\textwidth}
         \centering
         \includegraphics[width=\textwidth]{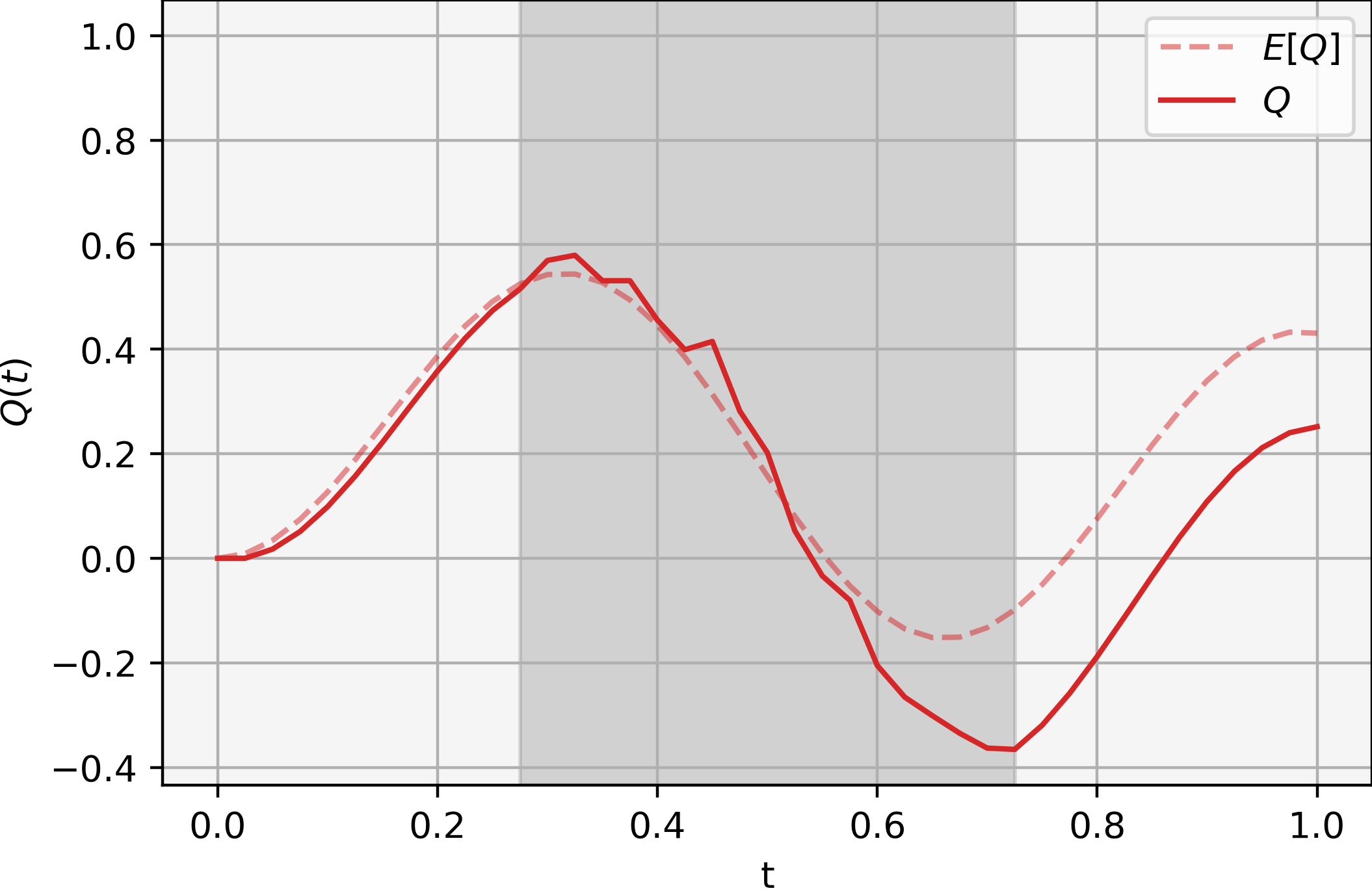}
         \caption{Supply realization.}
         \label{fig:Q2}
     \end{subfigure}
        \caption{Two supply realizations of \eqref{eq:SDE supply Numerical}. The grey window highlights the times where noise operates.}
        \label{fig:Q 2 samples}
\end{figure}
We select $L(x,v)=\frac{1}{2} \left( x - 1\right)^2 + \frac{1}{2} v^2$ and $u_T\left(x\right) = \frac{1}{2 e}\left(x-1\right)^2$. Figure \ref{fig:Estimates during training} shows the evolution of the a posteriori estimate in Theorem \ref{Thm:Main result}. The balance error achieves enough accuracy and slightly oscillates with a decreasing tendency. The optimality error also exhibits a decreasing trend, but accuracy does not improve, suggesting testing other combinations of training and discretization parameters.

\begin{figure}[htp]
     \centering
     \begin{subfigure}[t]{0.235\textwidth}
         \centering
         \includegraphics[width=\textwidth]{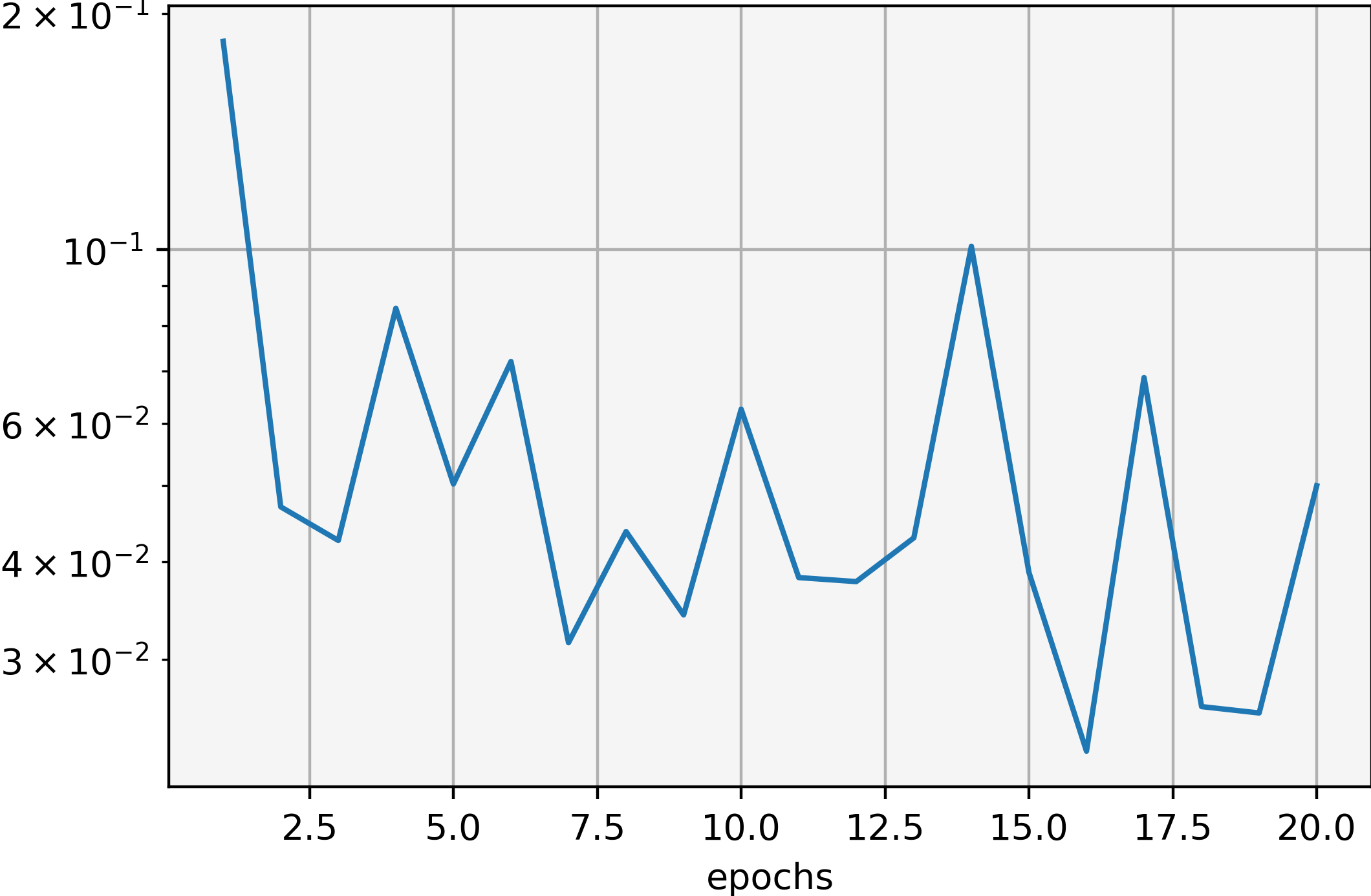}
         \caption{$\mbox{MSE}(\epsilon_B)$ (balance error).}         
     \end{subfigure}
     \hfill
     \begin{subfigure}[t]{0.235\textwidth}
         \centering
         \includegraphics[width=\textwidth]{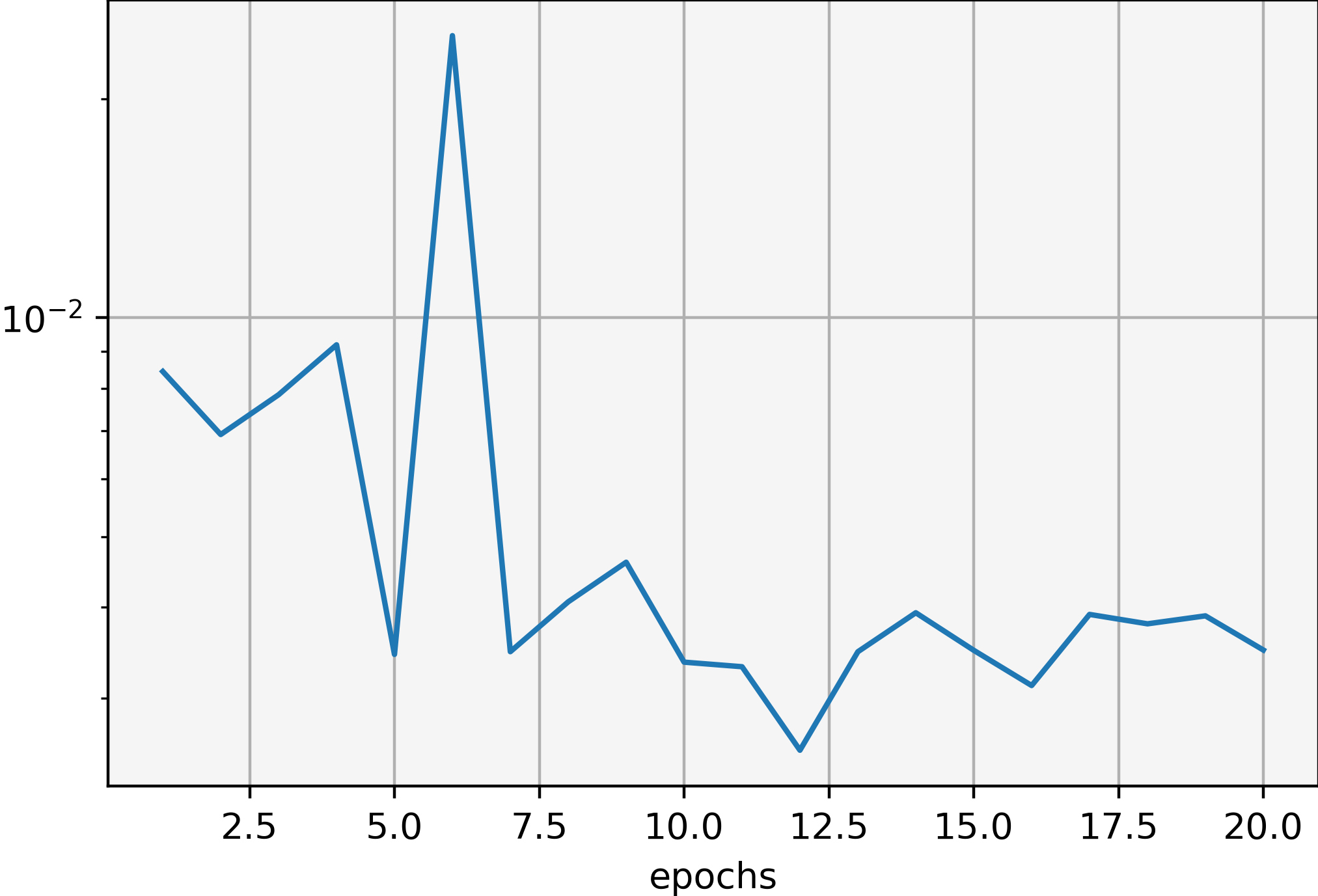}
		\caption{$\mbox{MSE}(\epsilon_H)$ (optimality error).}         
     \end{subfigure}
          \caption{Evolution of the a posteriori estimates during training.}
        \label{fig:Estimates during training}
\end{figure}
Furthermore, we use the analytic solution derived in \cite{gomes2021randomsupply} to verify the price approximation's accuracy. In the linear-quadratic framework, the price follows from the SDE  system 
\begin{equation}\label{eq:System price LQ}
	\begin{cases} 
	\d Q(t) =	(b^S(t)-Q_t) \d t +\sigma^S(t) \d W(t), 
	\\
	Q(0)=0,
	\\
	\d \overline{X}(t) = Q(t) \d t, 
	\\
	\overline{X}(0) =\overline{m}_0,
	\\
	\d \varpi(t) =	\left(\overline{X}(t) -  b^S(t) +Q(t) - 1 \right) \d t  
	\\
	\quad \quad \quad \quad - \frac{a_2^3(t) + 1}{a_2^4(t) +1} \sigma^S(t) \d W(t), 
	\\
	\varpi(0) =w_0.
	\end{cases}
\end{equation}
The value $w_0$ and the functions $a_2^3$ and $a_2^4$ are determined by a certain system of ordinary differential explicitly solvable.  
Figure \ref{fig:price 2 samples} shows the corresponding price approximation and exact price (obtained from \eqref{eq:System price LQ}) for the two supply realizations of Figure \ref{fig:Q 2 samples}. The decreasing trend in the errors observed Figure \ref{fig:Estimates during training} is reflected in the precise approximation observed in Figure \ref{fig:price 2 samples}. Notice the effect of the error in the time window $[0.25,0.75]$, which decreases, as expected, the accuracy of the approximation compared to the time region $[0,0.25]$, where no noise is applied.


\begin{figure}[htp]
     \centering
     \begin{subfigure}[t]{0.235\textwidth}
         \centering
         \includegraphics[width=\textwidth]{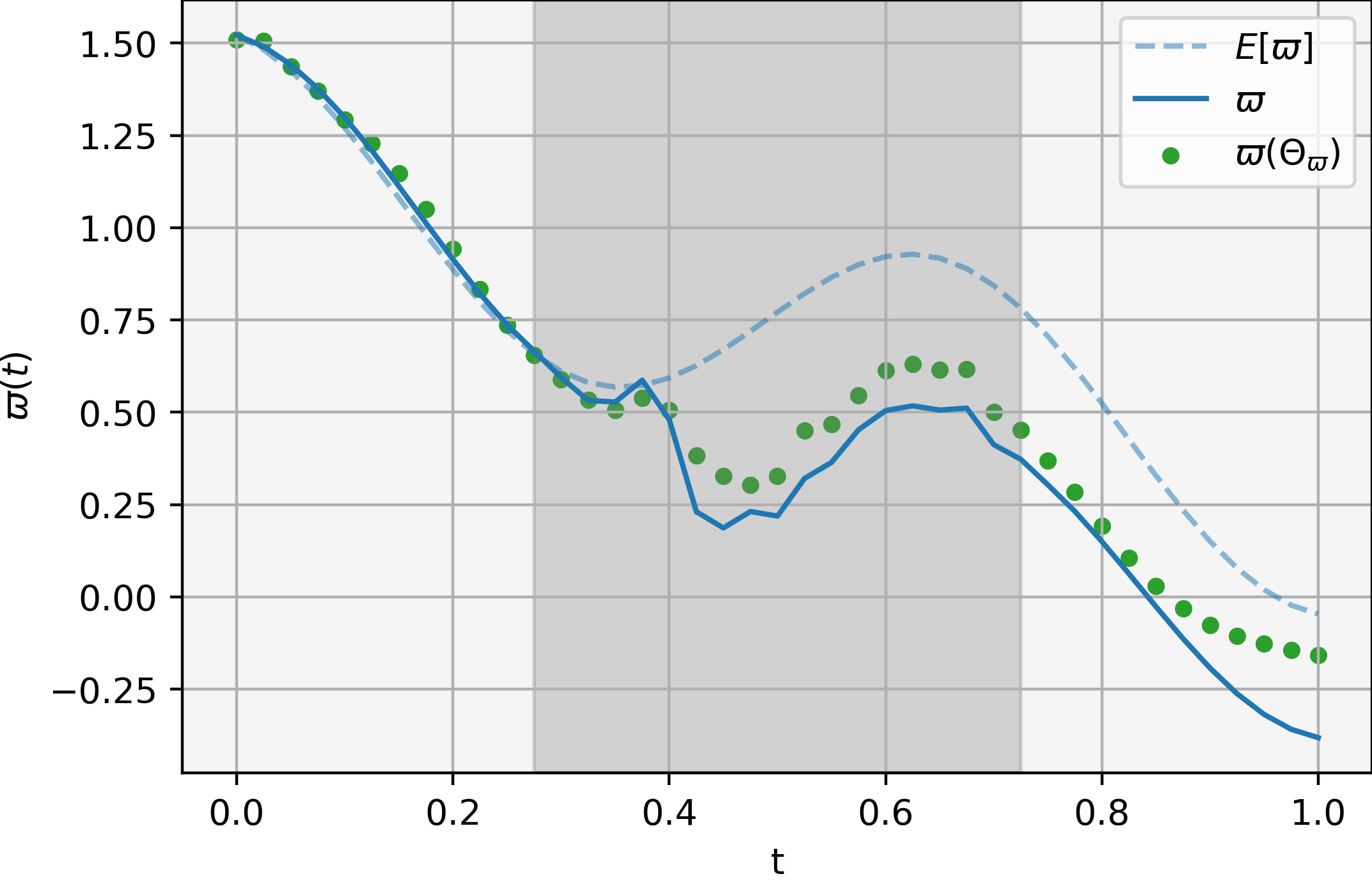}
         \caption{Price realization for Figure \ref{fig:Q1}.}         
     \end{subfigure}
     \hfill
     \begin{subfigure}[t]{0.235\textwidth}
         \centering
         \includegraphics[width=\textwidth]{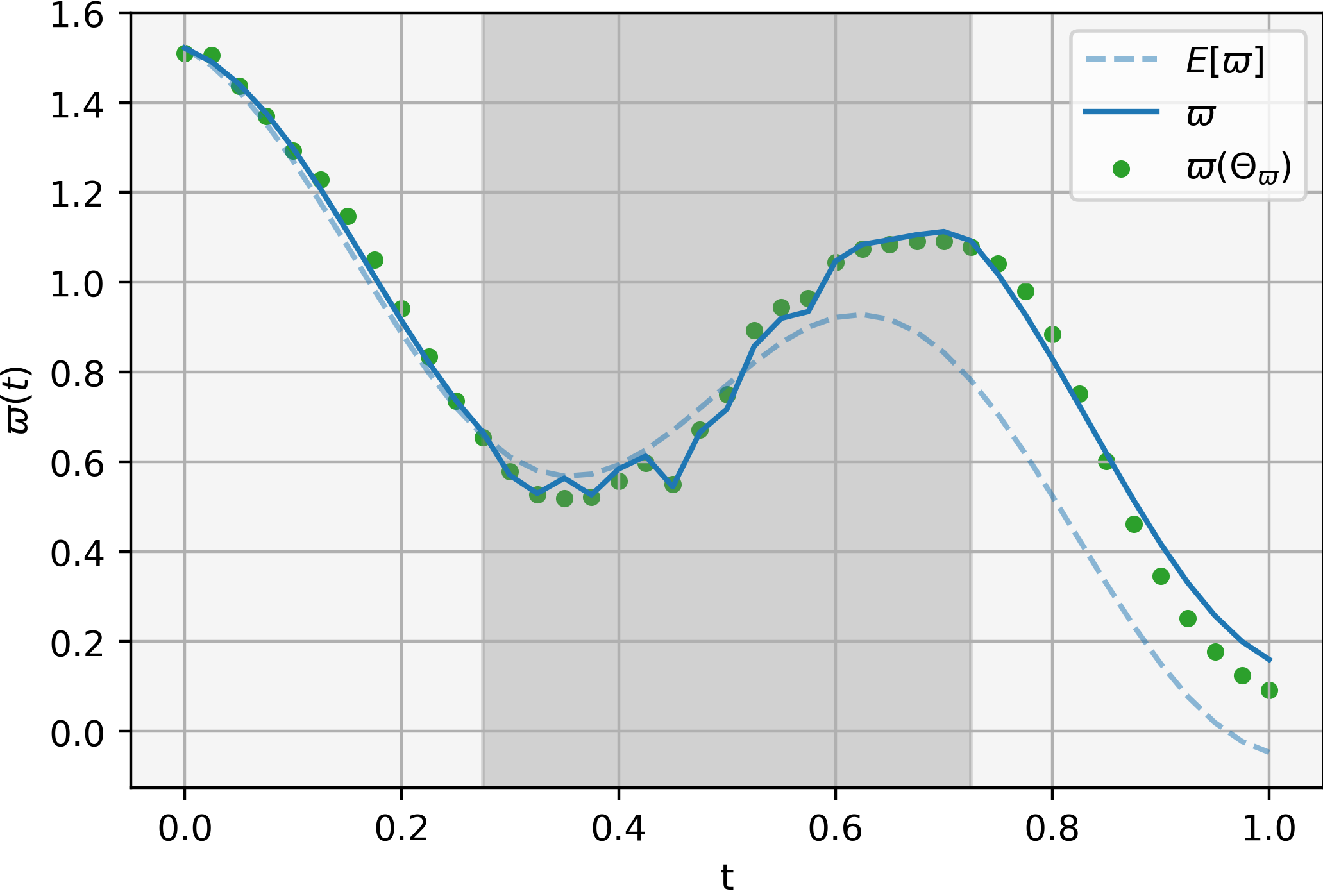}
         \caption{Price realization for Figure \ref{fig:Q2}.}         
     \end{subfigure}
        \caption{Exact price and RNN approximation for Figure \ref{fig:Q 2 samples}. The grey window highlights the times where noise operates.}
        \label{fig:price 2 samples}
\end{figure}

As the figures show, the method has an excellent performance in approximating solutions in various noise regimes. Further research and refinements can enhance its efficiency, speed, and accuracy, leading to even more precise approximations in critical regions.
\section{Conclusions and further directions}
\label{sec:Conclusions}

We extend the ML approach introduced in \cite{MLpriceDet2022} for the deterministic setting to  the common noise scenario, utilizing RNN architectures to represent non-anticipating controls. As in the deterministic case, our approach demonstrates good accuracy and performance.

Future research could explore method robustness concerning RNN and discretization parameter variations, addressing the trade-off between computational cost and accuracy. Comprehensive experiments may identify optimal RNN layer and neuron quantities for specific supply dynamics. Advanced coding methods could further reduce computational costs while maintaining or enhancing accuracy. Extensions could also accommodate additional noise sources, such as jump processes.

\addtolength{\textheight}{-12cm}   


\section*{ACKNOWLEDGMENT}

D. Gomes  and J. Gutierrez were supported by King Abdullah University of Science and Technology (KAUST) baseline funds and KAUST OSR-CRG2021-4674.


\bibliographystyle{abbrv}
\bibliography{mfg.bib}

\end{document}